\documentclass[letterpaper, 10 pt, conference]{ieeeconf}  

\IEEEoverridecommandlockouts                              

\usepackage{graphics} 
\usepackage{subfig}
\usepackage{epsfig} 
\usepackage{amsmath} 
\usepackage{amssymb}  
\usepackage{mathrsfs}
\usepackage{url}
\usepackage{cite}
\usepackage{epstopdf}
\usepackage{hyperref}

\newtheorem{theorem}{Theorem}

\newtheorem{assumption}{Assumption}
\newtheorem{definition}{Definition}
\newtheorem{example}{Example}

\newtheorem{problem}{Problem}

\newcommand{\Real}{\mathbb R}

\title{\LARGE \bf Dynamically Stable 3D Quadrupedal Walking with Multi-Domain Hybrid System Models and Virtual Constraint Controllers*}

\author{Kaveh Akbari Hamed$^{1}$, Wen-Loong Ma$^{2}$, and  Aaron D. Ames$^{2}$
\thanks{*The work of K. Akbari Hamed is supported by the National Science Foundation (NSF) under Grant Number 1637704. The work of A. D. Ames is supported by the NSF under Grant Numbers 1544332, 1724457, and 1724464 as well as Disney Research LA. The content is solely the responsibility of the authors and does not necessarily represent the official views of the NSF.}
\thanks{$^{1}$K. Akbari Hamed is with the Department of Mechanical Engineering, Virginia Tech, Blacksburg, VA 24061 USA {\tt\small kavehakbarihamed@vt.edu}}
\thanks{$^{2}$W. Ma and A. D. Ames are with the Department of Mechanical and Civil Engineering, California Institute of Technology, Pasadena, CA 91125 USA {\texttt{\small{wma@caltech.edu}} and \texttt{\small{ames@cds.caltech.edu}}}}
}

\begin{document}

\maketitle
\thispagestyle{empty}
\pagestyle{empty}


\begin{abstract}
Hybrid systems theory has become a powerful approach for designing feedback controllers that achieve dynamically stable bipedal locomotion, both formally and in practice. This paper presents an analytical framework 1) to address multi-domain hybrid models of quadruped robots with high degrees of freedom, and 2) to systematically design nonlinear controllers that asymptotically stabilize periodic orbits of these sophisticated models. A family of parameterized virtual constraint controllers is proposed for continuous-time domains of quadruped locomotion to regulate holonomic and nonholonomic outputs. The properties of the Poincar\'e return map for the full-order and closed-loop hybrid system are studied to investigate the asymptotic stabilization problem of dynamic gaits. An iterative optimization algorithm involving linear and bilinear matrix inequalities is then employed to choose stabilizing virtual constraint parameters. The paper numerically evaluates the analytical results on a simulation model of an advanced $3$D quadruped robot, called GR Vision $60$, with $36$ state variables and $12$ control inputs. An optimal amble gait of the robot is designed utilizing the FROST toolkit. The power of the analytical framework is finally illustrated through designing a set of stabilizing virtual constraint controllers with $180$ controller parameters.
\end{abstract}


\vspace{-0.25em}
\section{INTRODUCTION}
\label{Introduction}
\vspace{-0.25em}

This paper establishes an analytical foundation to systematically design nonlinear controllers that asymptotically stabilize periodic orbits for multi-domain hybrid models of $3$D quadruped locomotion with high degrees of freedom. We present a family of virtual constraint controllers that regulate holonomic and nonholonomic outputs for different domains of locomotion. The paper presents a scalable algorithm to design stabilizing controllers for the full-order hybrid models of locomotion rather than simplified ones. The framework can ameliorate specific challenges in the design of nonlinear controllers for hybrid systems of quadruped robots arising from high dimensionality and underactuation.

\vspace{-0.25em}
\subsection{Related Work}
\label{Related Work}
\vspace{-0.25em}

Models of legged locomotion are hybrid with continuous-time domains representing the Lagrangian dynamics and discrete-time transitions representing the change in the physical and unilateral constraints
\cite{Grizzle_Asymptotically_Stable_Walking_IEEE_TAC,Chevallereau_Grizzle_3D_Biped_IEEE_TRO,Ames_RES_CLF_IEEE_TAC,Spong_Controlled_Symmetries_IEEE_TAC,Spong_Passivity_IEEE_RAM,Manchester_Tedrake_LQR_IJRR,Tedrake_L2_gain_ICRA,Guobiao_Zefran_IEEE_ICRA,Gregg_3D_Biped_IEEE_TRO,Gregg_Controlled_Reduction_IEEE_TAC,Byl_Tedrake_Approximate_optimal_control_ICRA,Hamed_Gregg_decentralized_control_IEEE_CST,Hamed_Grizzle_Event_Based_IEEE_TRO,Cheavallereau_Grizzle_RABBIT,Morriz_Grizzle_Hybrid_Invariant_Manifolds_IEEE_TAC,Poulakakis_Grizzle_SLIP_IEEE_TAC,Sreenath_Grizzle_HZD_Walking_IJRR,Collins_Ruina_Tedrake_Wisse_Efficient_Bipedal_Robots,Tedrake_Stochastic_policy_gradient_IROS,Byl01082009,Johnson_Burden_Koditschek,Burden_SIAM,Vasudevan2017,Park_Cheetah,Ioannis_ISS_2017}.
Steady-state walking locomotion can be considered as periodic solutions of  these multi-domain hybrid systems. State-of-the-art controller design methods that address hybrid nature of models of legged locomotion are developed based on hybrid reduction \cite{Ames_HybridReduction_Original_Paper,Ames_Sastry_Hybrid_Reduction_CDC,Ames_Book_Chapter,Gregg_Spong_IJRR}, controlled symmetries \cite{Spong_Controlled_Symmetries_IEEE_TAC}, transverse linearization \cite{Manchester_Tedrake_LQR_IJRR,Shiriaev_Transverse-Linearization_IEEE_TAC}, and hybrid zero dynamics (HZD) \cite{Westervelt_Grizzle_Koditschek_HZD_IEEE_TRO,Morriz_Grizzle_Hybrid_Invariant_Manifolds_IEEE_TAC}. From these methods, transverse linearization and HZD can address underactuation. The notion of HZD was introduced in \cite{Westervelt_Grizzle_Koditschek_HZD_IEEE_TRO} to design feedback controllers that can explicitly accommodate underactuation in bipedal robots and move beyond flat-footed walking gaits arising from the Zero Moment Point (ZMP) criterion \cite{FRI,Vukobratovic_Book}. In the HZD method, a set of output functions, referred to as virtual constraints \cite{Maggiore_Virtual_Constraints_IEEE_TAC,Mohammadi_Automatica}, is defined and enforced by input-output (I-O) linearizing feedback controllers \cite{Isidori_Book}. HZD-based controllers have been validated numerically and experimentally for $2$D and $3$D bipedal robots \cite{Cheavallereau_Grizzle_RABBIT,Martin_Schmiedeler_IJRR,Ames_DURUS_TRO,Ramezani_Hurst_Hamed_Grizzle_ATRIAS_ASME,Park_Grizzle_Finite_State_Machine_IEEE_TRO,Byl_HZD}, $2$D and $3$D powered prosthetic legs \cite{Gregg_Virtual_Constraints_Powered_Prosthetic_IEEE_TRO,multicontactprosthetic_02}, exoskeletons \cite{Grizzle_Decentralized}, monopedal \cite{Poulakakis_Grizzle_SLIP_IEEE_TAC}, and simple (i.e., reduced-order) models of quadruped robots \cite{Ioannis2016}.

\begin{figure}[!t]
\centering
\vspace{0.4em}
\includegraphics[width=2.5in]{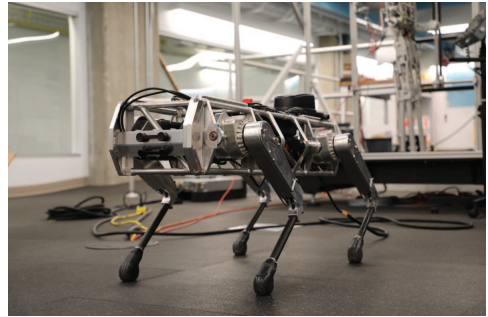}
\caption{Vision $60$, a $3$D quadruped robot with $18$ DOFs, designed and manufactured by Ghost Robotics \cite{Ghost_Robotics}.}
\label{GR_vision_60}
\vspace{-1.0em}
\end{figure}

\vspace{-0.4em}
\subsection{Motivation}
\label{Motivation}
\vspace{-0.25em}

The extension of the HZD approach to design nonlinear controllers for full-order and multi-domain hybrid models of quadruped robots is a significant challenge. In particular, the mechanical systems of these machines may form closed kinematic chains during multi-contact domains of locomotion for which at least two legs are in contact with the ground. This complicates the design procedure of virtual constraints such that 1) the I-O linearization results in a full-rank decoupling matrix and 2) the corresponding zero dynamics manifold becomes nontrivial. In addition to this, we have observed that the proper selection of the virtual constraints (i.e., output functions) can drastically affect the stability properties of walking gaits \cite{Hamed_Buss_Grizzle_BMI_IJRR}. The most basic tool to investigate the stability of hybrid periodic orbits is the Poincar\'e sections analysis \cite{Grizzle_Asymptotically_Stable_Walking_IEEE_TAC,Ioannis_ISS_2017,Haddad_Hybrid_Book}. One drawback of the Poincar\'e sections approach is the lack of closed-form expressions for the Poincar\'e return map that complicates the design of stabilizing virtual constraint controllers. To overcome this difficulty, our previous work \cite{Hamed_Buss_Grizzle_BMI_IJRR,Hamed_Gregg_decentralized_control_IEEE_CST,Hamed_Grizzle_NAHS} presented an iterative optimization algorithm involving bilinear matrix inequalities (BMIs) to systematically choose stabilizing virtual constraints. This algorithm was successfully employed for the centralized/decentralized feedback control design of autonomous bipedal robots with up to $13$ degrees of freedom (DOFs) \cite{Hamed_Buss_Grizzle_BMI_IJRR} and powered prosthetic legs \cite{Hamed_Gregg_decentralized_control_IEEE_CST}. In this paper, we aim to answer these fundamental questions: 1) how can we present holonomic and nonholonomic virtual constraint controllers for single- and multi-contact domains of quadruped locomotion, 2) how can we address the asymptotical stabilization problem of gaits for high-dimensional hybrid models of quadruped robots, and 3) can the BMI algorithm look for optimal and stabilizing HZD-based controllers of quadruped robots in a scalable manner?

\vspace{-0.25em}
\subsection{Goals, Objectives, and Contributions}
\label{Goals, Objectives, and Contributions}
\vspace{-0.25em}

The \textit{primary goal} of this paper is to establish an analytical foundation to 1) address multi-domain hybrid models of quadruped robots with high degrees of freedom, and 2) systematically design HZD-based controllers that stabilize hybrid periodic orbits. This goal will be achieved through the following objectives and \textit{contributions}.
1) A family of parameterized HZD-based controllers is presented for single- and multi-contact domains of quadruped locomotion. These controllers are utilized to zero a combination of parameterized holonomic and nonholonomic outputs; 2) The properties of the full-order Poincar\'e return map are investigated to address the asymptotic stabilization problem of dynamic gaits; 3) The stabilization problem of multi-domain hybrid periodic orbits is then translated into an iterative BMI optimization problem that solves for the virtual constraint parameters; 4) An optimal and dynamic amble gait is designed for an advanced quadruped robot, called Vision 60 (see Fig. \ref{GR_vision_60}), with $36$ state variables and $12$ control inputs. The motion planning algorithm is based on a nonlinear programming problem that is effectively solved using the FROST toolkit \cite{FROTS}; and 5) Analytical results are finally confirmed through designing a stabilizing HZD-based controller with $180$ parameters.


\section{HYBRID MODEL OF LOCOMOTION}
\label{Hybrid Model of Walking}
\vspace{-0.25em}

\subsection{Vision $60$}
\label{GRvision60_section}
\vspace{-0.25em}

Vision $60$ is a mid-sized tele-op and autonomous all-terrain ground drone designed and manufactured by Ghost Robotics \cite{Ghost_Robotics} for research markets (see Fig. \ref{GR_vision_60}). It weighs approximately $20$ (kg) and supports a total payload of $14$ (kg). Vision $60$ has $18$ DOFs of which $12$ leg DOFs are actuated. In particular, each leg of the robot consists of a $1$ DOF actuated knee joint with pitch motion and a $2$ DOF actuated hip joint with pitch and roll motions. The remaining $6$ DOFs are associated with the translational and rotational movements of the torso. The robot can transverse a range of unstructured terrains and substrates, and even stairs. 



\vspace{-0.25em}
\subsection{Robot Model}
\label{Robot Model}
\vspace{-0.25em}

To describe the configuration of the robot, we make use of a floating base coordinates system. For this purpose, let us attach a body frame $R_{b}$ to the base of the robot. The Cartesian coordinates of the origin of this fame with respect to an inertial world frame, denoted by $R_{0}$, can be given by $p_{b}\in\Real^{3}$ (see Fig. \ref{float_base_model}). Furthermore, the orientation of $R_{b}$ with respect to $R_{0}$ is expressed by $\phi_{b}\in\textrm{SO}(3)$. Next let us suppose that $q_{\textrm{body}}\in\mathcal{Q}_{\textrm{body}}$ denotes a set of coordinates to describe the body (shape) of the robot. The generalized coordinates vector is then taken as $q:=\textrm{col}(p_{b},\phi_{b},q_{\textrm{body}})\in\mathcal{Q}:=\Real^{3}\times\textrm{SO}(3)\times\mathcal{Q}_{\textrm{body}}$, where $\mathcal{Q}$ represents the configuration space. For future purposes, we define $n_{q}:=\textrm{dim}(q)$ as the number of degrees of freedom for the floating base model. The state vector of the mechanical system is finally taken as $x:=\textrm{col}(q,\dot{q})\in\textrm{T}\mathcal{Q}$, in which $\textrm{T}\mathcal{Q}$ denotes the tangent bundle of $\mathcal{Q}$.

\begin{figure}[!t]
\centering
\includegraphics[width=2.5in]{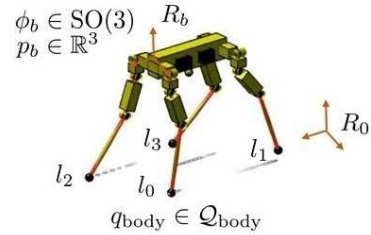}
\vspace{-1em}
\caption{Floating base model of the robot with the associated configuration variables.}
\label{float_base_model}
\vspace{-1.5em}
\end{figure}


\subsection{Hybrid Systems Formulation for Locomotion}
\label{Hybrid Systems Formulation for Walking}

The open-loop hybrid model of quadruped locomotion can be given by the following tuple
\begin{equation}\label{open_loop_hybrid_model}
\mathscr{HL}^{\textrm{ol}}=\left(\Lambda,\mathcal{X},\mathcal{U},\mathcal{D},\mathcal{S},\Delta,FG\right),
\end{equation}
where $\Lambda:=\left(\mathcal{V},\mathcal{E}\right)$ represents a \textit{directed cycle} with the vertices set $\mathcal{V}$ and the edges $\mathcal{E}\subseteq\mathcal{V}\times\mathcal{V}$ (see Fig. \ref{hybrid_model_illustration}). In this formulation, the vertices represent the continuous-time dynamics of locomotion, referred to as \textit{domains} or \textit{phases}. The evolution of the system during each domain is described by ordinary differential equations (ODEs) arising from the Lagrangian dynamics. The edges denote the discrete-time transitions between two continuous-time domains that arise from the change in the number of physical constraints. The discrete-time transitions are further supposed to be instantaneous. In this paper, we assume that $\mu:\mathcal{V}\rightarrow\mathcal{V}$ denotes the \textit{index of the next domain function} for the studied locomotion pattern. Using this notation, the set of edges can be expressed as $\mathcal{E}:=\left\{e=\left(v\rightarrow\mu(v)\right)\right\}_{v\in\mathcal{V}}$. The \textit{set of state manifolds} for the graph \eqref{open_loop_hybrid_model} is represented by $\mathcal{X}:=\left\{\mathcal{X}_{v}\right\}_{v\in\mathcal{V}}$, in which $\mathcal{X}_{v}\subseteq\Real^{2n_{q}}$ denotes the \textit{state manifold} for the vertex (i.e., domain) $v\in\mathcal{V}$. The \textit{set of admissible controls} is also given by $\mathcal{U}:=\left\{\mathcal{U}_{v}\right\}_{v\in\mathcal{V}}$, where $\mathcal{U}_{v}\subseteq\Real^{m}$ represents the \textit{set of admissible control inputs} for the domain $v\in\mathcal{V}$ and some positive integer $m$. We remark that for the Vision $60$, $\mathcal{X}_{v}\subseteq\Real^{36}$ and $\mathcal{U}_{v}\subseteq\Real^{12}$. $\mathcal{D}:=\left\{\mathcal{D}_{v}\right\}_{v\in\mathcal{V}}$ denotes the \textit{set of domain of admissibility}, in which $\mathcal{D}_{v}\subseteq\mathcal{X}_{v}\times\mathcal{U}_{v}$ is a smooth submanifold of $\Real^{2n_{q}}\times\Real^{m}$. $FG:=\{\left(f_{v},g_{v}\right)\}_{v\in\mathcal{V}}$ represents the \textit{set of control systems}, where $\left(f_{v},g_{v}\right)$ is a control system on $\mathcal{D}_{v}$. In particular, the evolution of the continuous-time domain $v\in\mathcal{V}$ is expressed by the input-affine state equation $\dot{x}=f_{v}(x)+g_{v}(x)\,u$ for $(x,u)\in\mathcal{D}_{v}$ with $f_{v}$ and columns of the $g_{v}$ matrix being \textit{smooth} (i.e., $\mathcal{C}^{\infty}$) on $\mathcal{X}_{v}$. The \textit{set of guards} for the hybrid system \eqref{open_loop_hybrid_model} is then represented by $\mathcal{S}:=\left\{\mathcal{S}_{e}\right\}_{e\in\mathcal{E}}$ on which the discrete-time transition $v\rightarrow\mu(v)$ occurs when the state and control trajectories $\left(x(t),u(t)\right)$ intersect the guard $\mathcal{S}_{v\rightarrow\mu(v)}\subset\mathcal{D}_{v}$. $\Delta:=\{\Delta_{e}\}_{e\in\mathcal{E}}$ is finally a set of \textit{reset laws} to describe discrete-time transitions, where $\Delta_{v\rightarrow\mu(v)}$ is a \textit{smooth} discrete-time system represented by $x^{+}=\Delta_{v\rightarrow\mu(v)}\left(x^{-}\right)$ for $v\in\mathcal{V}$. In our notation, $x^{-}(t):=\lim_{\tau\nearrow t}x(\tau)$ and $x^{+}(t):=\lim_{\tau\searrow t}x(\tau)$ denote the left and right limits of the state trajectory $x(t)$, respectively.

\begin{figure}[!t]
\centering
\vspace{0.2em}
\includegraphics[width=3.2in]{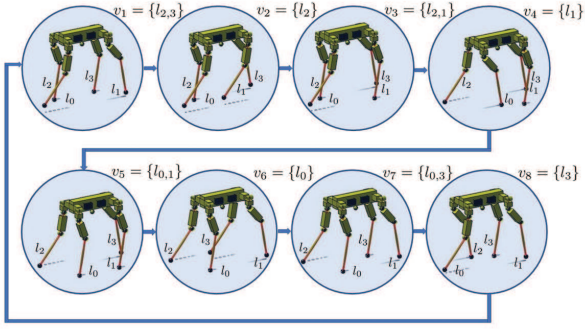}
\vspace{-0.75em}
\caption{Illustration of the multi-domain hybrid model of $3$D amble gait. Continuous-time domains and discrete-time transitions are represented by the vertices and edges of a directed cycle $\Lambda=\left(\mathcal{V},\mathcal{E}\right)$, respectively.}
\label{hybrid_model_illustration}
\vspace{-1.5em}
\end{figure}

\begin{example}[Amble Gait]\label{Amble_gait}
In this example, we consider an eight-domain directed cycle illustrating a typical amble gait shown in Fig. \ref{hybrid_model_illustration}. The legs of the robot are enumerated as $\{0,1,2,3\}$. The directed cycle $\Lambda=\left(\mathcal{V},\mathcal{E}\right)$ for this gait then consists of eight vertices and edges. In particular, $\mathcal{V}=\{l_{2,3},l_{2},l_{2,1},l_{1},l_{0,1},l_{0},l_{0,3},l_{3}\}$ and $\mathcal{E}=\{l_{2,3}\rightarrow{}l_{2},l_{2}\rightarrow{}l_{2,1},l_{2,1}\rightarrow{}l_{1},l_{1}\rightarrow{}l_{0,1},l_{0,1}\rightarrow{}l_{0},l_{0}\rightarrow{}l_{0,3},l_{0,3}\rightarrow{}l_{3},l_{3}\rightarrow{}l_{2,3}\}$,
where $l_{i}$, $i\in\{0,1,2,3\}$ denotes the domains for which the leg $i$ is in contact with the ground. Furthermore, $l_{i,j}$, $i\neq{}j\in\{0,1,2,3\}$ represents the domains in which the legs $i$ and $j$ are simultaneously in contact with the ground.
\end{example}


\vspace{-0.25em}
\subsection{Continuous-Time Dynamics}
\label{Continuous-Time Dynamics}
\vspace{-0.25em}

During the continuous-time domain $v\in\mathcal{V}$, we assume that $\mathcal{C}_{v}$ represents the \textit{indexing set of holonomic constraints} defined on $\mathcal{D}_{v}$. In particular, the holonomic physical constraints are expressed as $\eta_{v}(q)=0$, where $\eta_{v}:=\{\eta_{c}\}_{c\in\mathcal{C}_{v}}\in\Real^{n_{v}}$. The associated velocity constraints can be given by $J_{v}(q)\,\dot{q}=0$, in which $J_{v}(q):=\frac{\partial \eta_{v}}{\partial q}(q)\in\Real^{n_{v}\times{}n_{q}}$ denotes the corresponding Jacobian matrix that is assumed to be full-rank. The evolution of the mechanical system during the continuous-time domain $v$ is then expressed as the following second-order ODE arising from the Euler-Lagrange equations and principle of virtual work
\begin{alignat}{6}
&D(q)\,\ddot{q}+C\left(q,\dot{q}\right)\dot{q}+G(q)&&=B\,u+J_{v}^\top(q)\,\lambda\nonumber\\
&J_{v}(q)\,\ddot{q}+\frac{\partial}{\partial q}\left(J_{v}(q)\,\dot{q}\right)\dot{q}&&=0,\label{EL_equations}
\end{alignat}
where $D(q)\in\Real^{n_{q}\times{n_{q}}}$ denotes the positive definite mass-inertia matrix, $C(q,\dot{q})\,\dot{q}\in\Real^{n_{q}}$ represents the Coriolis and centrifugal terms, and $G(q)\in\Real^{n_{q}}$ contains the gravitational terms. The input distribution matrix and Lagrange multipliers are given by $B\in\Real^{n_{q}\times{}m}$ with the property $\textrm{rank}\,B=m$ and $\lambda\in\Real^{n_{v}}$, respectively. By eliminating the Lagrange multipliers from \eqref{EL_equations}, one can obtain the following \textit{constrained dynamics} for the domain $v$
\begin{equation}\label{constrained_dynamics}
D(q)\,\ddot{q}+F_{v}\left(q,\dot{q}\right)=T_{v}(q)\,u,
\end{equation}
in which $F_{v}:=\textrm{proj}_{v}\,F+J_{v}^\top\,(J_{v}\,D^{-1}\,J_{v}^\top)^{-1}\frac{\partial}{\partial q}(J_{v}\,\dot{q})\dot{q}$, $F=C(q,\dot{q})\,\dot{q}+G(q)$, $T_{v}:=\textrm{proj}_{v}\,B$, and $\textrm{proj}_{v}:=I-J_{v}^\top(J_{v}\,D^{-1}\,J_{v}^\top)^{-1}J_{v}\,D^{-1}$.
The equations of motion in \eqref{constrained_dynamics} can be rewritten in the state equation form $\dot{x}=f_{v}(x)+g_{v}(x)\,u$ for which the state manifold is given by
\begin{equation*}
\mathcal{X}_{v}:=\left\{x=\textrm{col}\left(q,\dot{q}\right)\in\textrm{T}\mathcal{Q}\,|\,\eta_{v}(q)=0,\,J_{v}(q)\,\dot{q}=0\right\}.
\end{equation*}
According to the construction procedure, $\mathcal{X}_{v}$ is forward-invariant under the flow of the state equation. Furthermore, $\textrm{dim}(\mathcal{X}_{v})=2(n_{q}-n_{v})$.


\vspace{-0.25em}
\subsection{Discrete-Time Dynamics}
\label{Discrete-Time Dynamics}
\vspace{-0.25em}

Discrete-time transitions occur when there is a change in physical constraints. If one of the existing contacts breaks during the discrete-transition $e=(v\rightarrow\mu(v))$, the discrete-time dynamics are simply taken as the identity map, i.e., $x^{+}=\Delta_{e}\left(x^{-}\right):=x^{-}$ to preserve the continuity of position and velocity. However, if there is a new contact point, the state of the mechanical system would undergo an abrupt change in the velocity components according to the instantaneous impact model between two rigid bodies \cite{Hurmuzlu_Impact}. More precisely, the conservation of the generalized momentum during the infinitesimal period of the impact results in
\begin{equation}\label{impact}
D(q)\,\dot{q}^{+}-D(q)\,\dot{q}^{-}=J_{\mu(v)}^\top\,\delta\lambda,\,\,\,J_{\mu(v)}(q)\,\dot{q}^{+}=0,
\end{equation}
in which $\dot{q}^{-}$ and $\dot{q}^{+}$ represent the generalized velocity right before and after the impact, respectively, and $\delta\lambda$ denotes the intensity of the impulsive Lagrange multipliers. Using \eqref{impact} and the continuity of position (i.e., $q^{+}=q^{-}$), one can express the discrete-time mapping as $x^{+}=\Delta_{e}(x^{-})$.


\vspace{-0.25em}
\subsection{Solutions and Periodic Orbits}
\label{Solutions and Periodic Orbits}
\vspace{-0.25em}

Solutions of the open-loop hybrid model \eqref{open_loop_hybrid_model} are constructed by piecing together the flows of the continuous-time domains such that the discrete-time transitions occur when the state and control trajectories cross the switching manifolds. To make this concept more precise, we parameterize the solutions by the continuous time $t$ as well as the vertex number $v$ and present the following definition.

\begin{definition}[Solutions]
$(x,u):[0,t_{f})\times\mathcal{V}\rightarrow\mathcal{D}, t_{f}\in\Real_{>0}\cup\{\infty\}$ is said to be a \textit{solution} for \eqref{open_loop_hybrid_model} if
\begin{enumerate}
\item $x(t,v)$ and $u(t,v)$ are right continuous on $[0,t_{f})$ for every $v\in\mathcal{V}$;
\item The left and right limits $x^{-}(t,v):=\lim_{\tau\nearrow t}x(\tau,v)$, $u^{-}(t,v):=\lim_{\tau\nearrow t}u(\tau,v)$, $x^{+}(t,v):=\lim_{\tau\searrow t}x(\tau,v)$, and $u^{+}(t,v):=\lim_{\tau\searrow t}u(\tau,v)$,  exist for every $t\in(0,t_{f})$ and $v\in\mathcal{V}$; and
\item There exists a closed discrete subset $\mathcal{T}:=\{t_{0}<t_{1}<t_{2}<\cdots\}\subset[0,t_{f})$, referred to as the \textit{switching times}, such that (a) for every $(t,v)\in[0,t_{f})\setminus\mathcal{T}\times\mathcal{V}$, $x(t,v)$ is differentiable with respect to $t$, $\frac{\partial x}{\partial t}(t,v)=f_{v}(x(t,v))+g_{v}(x(t,v))\,u(t,v)$, $(x^{-}(t,v),u^{-}(t,v))\notin\mathcal{S}_{v\rightarrow\mu(v)}$, and (b) for $t=t_{j}\in\mathcal{T}$, $(x^{-}(t_{j},v),u^{-}(t_{j},v))\in\mathcal{S}_{v\rightarrow\mu(v)}$, $x^{+}(t_{j},\mu(v))=\Delta_{v\rightarrow\mu(v)}(x^{-}(t_{j},v))$.
\end{enumerate}
\end{definition}

\begin{assumption}[Periodic Solution]\label{periodic_orbit_assumption}
There exist (i) a \textit{nominal} solution $(x^{\star},u^{\star}):[0,\infty)\times\mathcal{V}\rightarrow\mathcal{D}$ to \eqref{open_loop_hybrid_model} and (ii) a \textit{fundamental period} $T^{\star}>0$ such that $x^{\star}(t+T^{\star},v)=x^{\star}(t,v)$  and $u^{\star}(t+T^{\star},v)=u^{\star}(t,v)$ for every $(t,v)\in\Real_{\geq0}\times\mathcal{V}$. The corresponding \textit{periodic orbit} is defined as
\begin{equation*}
\mathcal{O}:=\cup_{v\in\mathcal{V}}\mathcal{O}_{v}:=\left\{x=x^{\star}(t,v)\,|\,(t,v)\in[0,T^{\star})\times\mathcal{V}\right\},
\end{equation*}
where $\mathcal{O}_{v}$ is the projection of $\mathcal{O}$ onto the state manifold $\mathcal{X}_{v}$.
\end{assumption}


\vspace{-0.25em}
\section{FAMILY OF PARAMETERIZED VIRTUAL CONSTRAINT CONTROLLERS}
\label{FAMILY OF PARAMETERIZED VIRTUAL CONSTRAINT CONTROLLERS}
\vspace{-0.25em}

The objective of this section is to present a family of parameterized virtual constraint controllers that asymptotically stabilize dynamic gaits for multi-domain hybrid models of quadruped locomotion. In Section \ref{STABILIZATION PROBLEM}, we will show that the stability of the gaits depends on the proper selection of the virtual constraints and that is the reason to parameterize the constraints by a set of finite-dimensional and adjustable parameters $\xi_{v}\in\Xi_{v}$. Here, $\xi_{v}$ represents the \textit{virtual constraint parameters} for the domain $v\in\mathcal{V}$ and $\Xi_{v}\subset\Real^{p_{v}}$ denotes the corresponding \textit{set of admissible parameters} for some positive integer $p_{v}>2n_{q}$. The virtual constraint controllers are then assumed to be time-invariant and nonlinear state feedback laws, based on input-output linearization \cite{Isidori_Book}, to zero a combination of parameterized holonomic and nonholonomic outputs. The BMI algorithm of Section \ref{BMI ALGORITHM} will optimize these parameters for the asymptotic stabilization of the gait.

Virtual constraints are defined as output functions for the continuous-time domains of hybrid models of walking to coordinate the links of robots within a stride
\cite{Maggiore_Virtual_Constraints_IEEE_TAC,Mohammadi_Automatica,Cheavallereau_Grizzle_RABBIT,Martin_Schmiedeler_IJRR,Ames_DURUS_TRO,Ramezani_Hurst_Hamed_Grizzle_ATRIAS_ASME,Park_Grizzle_Finite_State_Machine_IEEE_TRO,Byl_HZD,Gregg_Virtual_Constraints_Powered_Prosthetic_IEEE_TRO,Gregg_Toward_Biomimetic_Control_IEEE_CST,multicontactprosthetic_02,Grizzle_Decentralized,Poulakakis_Grizzle_SLIP_IEEE_TAC}.
In this paper, for single-contact domains of locomotion with only one lege on the ground, we regulate a holonomic virtual constraint for position tracking purposes. For multi-contact domains with at least two legs in contact with the ground, we propose a combination of holonomic and nonholonomic virtual constraints. The holonomic constraint is again utilized for position tracking purposes. In addition, if there are enough admissible actuators present, the forward velocity of the robot can be controlled as well through zeroing a nonholonomic constraint. The idea of using nonholonomic constraints has been motivated by the velocity modulating outputs in \cite{Ames_Nonholonomic,Ames_Human_Inspired_IEE_TAC,Ames_DURUS_TRO} to regulate the speed of the mechanical system. To present the main idea, we define the concept of a phasing variable.

\begin{assumption}[Phasing Variable]\label{phasing_variable_assumption}
There exists a real-valued function $\tau:\mathcal{X}\times\mathcal{V}\rightarrow\Real$, referred to as the \textit{phasing variable}, which is (i) $\mathcal{C}^{\infty}$ with respect to $x$ and (ii) strictly increasing function of time along the orbit $\mathcal{O}_{v}$ for every $v\in\mathcal{V}$.
\end{assumption}

The phasing variable replaces time, which is a key to obtaining time-invariant controllers that realize orbital stability of $\mathcal{O}$. More precisely, one can express the desired evolution of the state variables on $\mathcal{O}_{v}$ in terms of $\tau$ rather $t$ as $x^{\star}(\tau,v)$. For the purpose of this paper, we assume that phasing variables are taken as holonomic quantities. Now we are in a position to present the virtual constraint controllers for multi- and single-contact domains.


\subsubsection{Multi-Contact Domains}
\label{Multi-Contact Domains}

For multi-contact domains $v\in\mathcal{V}$, we consider a smooth and \textit{parameterized} output function $y_{v}$ to be regulated as follows
\begin{equation}
y_{v}(x,\xi_{v}):=\textrm{col}\left(y_{1v}(x),y_{2v}(x,\xi_{v})\right)\in\Real^{1+m_{2v}},
\end{equation}
where $y_{1v}(x)\in\Real$ and $y_{2v}(x,\xi_{v})\in\Real^{m_{2v}}$ denote the relative degree one and relative degree two portions of the output, respectively, for some positive integer $m_{2v}$. The relative degree one (i.e., nonholonomic) component $y_{1v}$ is chosen to regulate the speed of the robot, i.e.,
\begin{equation}\label{velocity_output}
y_{1v}(x):=s\left(q,\dot{q}\right)-s^{\star}\left(\tau,v\right)\in\Real,
\end{equation}
in which $s(q,\dot{q}):=J_{s}(q)\,\dot{q}$ denotes the forward speed of a point on the robot and $s^{\star}(\tau,v)$ represents the desired evolution of $s$ on the orbit $\mathcal{O}_{v}$ in terms of the phasing variable $\tau$. The relative degree two (i.e., holonomic) component is then defined as the following parameterized output
\begin{equation}
y_{2v}\left(x,\xi_{v}\right):=H_{2v}\left(\xi_{v}\right)\left(q-q^{\star}\left(\tau,v\right)\right)\in\Real^{m_{2v}},
\end{equation}
where $H_{2v}(\xi_{v})\in\Real^{m_{2v}\times{}n_{q}}$ is a parameterized output matrix to be determined and $q^{\star}(\tau,v)$ represents the desired evolution of the configuration variables on $\mathcal{O}_{v}$ in terms of $\tau$. We remark that the output matrix $H_{2v}(\xi_{v})$ is parameterized by the virtual constraint parameters $\xi_{v}\in\Xi_{v}$. One typical way for this parameterization is to assume that $\xi_{v}$ forms the columns of $H_{2v}$, i.e., $\xi_{v}=\textrm{vec}(H_{2v})$, where ``$\textrm{vec}$'' denotes the matrix vectorization operator. The BMI algorithm of Section \ref{BMI ALGORITHM} will look for the optimal parameters $\xi_{v}$ to asymptotically stabilize the gait.

Using standard input-output linearization, the output dynamics become
\begin{equation}
\begin{bmatrix}
\dot{y}_{1v}\\
\ddot{y}_{2v}
\end{bmatrix}=A_{v}\left(x,\xi_{v}\right)u+b_{v}\left(x,\xi_{v}\right),
\end{equation}
where
\begin{alignat}{6}
&A_{v}\left(x,\xi_{v}\right)&&:=\begin{bmatrix}
\textrm{L}_{g_{v}}y_{1v}(x)\\
\textrm{L}_{g_{v}}\textrm{L}_{f_{v}}y_{2v}\left(x,\xi_{v}\right)
\end{bmatrix}\in\Real^{(1+m_{2v})\times{}m}\label{decoupling_matrix_01}\\
&b_{v}\left(x,\xi_{v}\right)&&:=\begin{bmatrix}
\textrm{L}_{f_{v}}y_{1v}(x)\\
\textrm{L}_{f_{v}}^{2}y_{2v}\left(x,\xi_{v}\right)
\end{bmatrix}\in\Real^{1+m_{2v}}.
\end{alignat}
Assuming $m_{2v}<m-1$ and having a full-rank decoupling matrix on an open neighborhood of $\mathcal{O}_{v}$, i.e., $\textrm{rank}\,A_{v}(x,\xi_{v})=1+m_{2v}$, one can employ a nonlinear feedback law $\Gamma_{v}:\mathcal{X}_{v}\times\Xi_{v}\rightarrow\mathcal{U}_{v}$ as follows
\begin{equation}\label{IO_lineariz_controller}
u=\Gamma_{v}\left(x,\xi_{v}\right):=-A_{v}^\top\left(A_{v}\,A_{v}^\top\right)^{-1}\left(b_{v}+w_{v}\right)
\end{equation}
to yield the output dynamics
\begin{equation}\label{output_dyn}
\begin{bmatrix}
\dot{y}_{1v}\\
\ddot{y}_{2v}
\end{bmatrix}=-w_{v}:=-\begin{bmatrix}
k_{p}\,y_{1v}\\
k_{p}\,y_{2v}+k_{d}\,\dot{y}_{2v}
\end{bmatrix}.
\end{equation}
Here, $k_{p}$ and $k_{d}$ are positive PD gains that exponentially stabilize the origin $(y_{1v},y_{2v},\dot{y}_{2v})=(0,0,0)$ for \eqref{output_dyn}. The feedback law \eqref{IO_lineariz_controller} also renders the \textit{parameterized zero dynamics manifold}
$\mathcal{Z}_{v}\left(\xi_{v}\right):=\{x\in\mathcal{X}_{v}\,|\,y_{1v}(x)=0,\,y_{2v}(x,\xi_{v})=\textrm{L}_{f_{v}}y_{2v}(x,\xi_{v})=0\}$
attractive and forward-invariant under the flow of the closed-loop system $\dot{x}=f_{v}^{\textrm{cl}}(x,\xi_{v})$, where $f_{v}^{\textrm{cl}}(x,\xi_{v}):=f_{v}(x)+g_{v}(x)\,\Gamma_{v}(x,\xi_{v})$. According to the construction procedure, $\textrm{dim}(\mathcal{Z}_{v})=2(n_{q}-n_{v}-m_{2v})-1$.


\subsubsection{Single-Contact Domains}
\label{Single-Contact Domains}

For single-contact domains $v\in\mathcal{V}$, we only consider parameterized holonomic output functions to be regulated as follows
\begin{equation*}
y_{v}\left(x,\xi_{v}\right)=y_{2v}\left(x,\xi_{v}\right):=H_{2v}\left(\xi_{v}\right)\left(q-q^{\star}\left(\tau,v\right)\right)\in\Real^{m},
\end{equation*}
for which $\textrm{dim}(y)=\textrm{dim}(u)=m$ and $H_{2v}(\xi_{v})\in\Real^{m\times{}n_{q}}$. Analogous to the analysis for the multi-contact domains, we can show that $\ddot{y}_{2v}=A_{v}(x,\xi_{v})\,u+b_{v}(x,\xi_{v})$, in which $A_{v}\left(x,\xi_{v}\right):=\textrm{L}_{g_{v}}\textrm{L}_{f_{v}}y_{2v}\left(x,\xi_{v}\right)\in\Real^{m\times{}m}$ and $b_{v}\left(x,\xi_{v}\right):=\textrm{L}_{f_{v}}^{2}y_{2v}\left(x,\xi_{v}\right)\in\Real^{m}$. Therefore, the input-output linearizing controller is taken as
\begin{equation}
u=\Gamma_{v}\left(x,\xi_{v}\right):=-A_{v}^{-1}\left(b_{v}+k_{d}\dot{y}_{2v}+k_{p}y_{2v}\right)
\end{equation}
that renders the parameterized zero dynamics manifold $\mathcal{Z}_{v}\left(\xi_{v}\right):=\{x\in\mathcal{X}_{v}\,|\,y_{2v}(x,\xi_{v})=\textrm{L}_{f_{v}}y_{2v}(x,\xi_{v})=0\}$ attractive and forward-invariant under the flow of the closed-loop continuous-time domain $\dot{x}=f_{v}^{\textrm{cl}}(x,\xi_{v})$. We also remark that $\textrm{dim}(\mathcal{Z}_{v})=2(n_{q}-n_{v}-m)$. For future purposes, we suppose that the family of parameterized controllers $\Gamma:=\{\Gamma_{v}\}_{v\in\mathcal{V}}$ satisfies the following assumption.

\begin{assumption}[Nominal Parameters]
There exist \textit{nominal} controller parameters $\xi_{v}^{\star}\in\Xi_{v}$, $v\in\mathcal{V}$ such that
\begin{equation}
\Gamma_{v}\left(x^{\star}\left(t,v\right),\xi^{\star}_{v}\right)=u^{\star}\left(t,v\right),\quad\forall{}t\in[0,T^{\star}),\,v\in\mathcal{V}.
\end{equation}
\end{assumption}


\section{STABILIZATION PROBLEM}
\label{STABILIZATION PROBLEM}
\vspace{-0.25em}

The objective of this section is to address the asymptotic stabilization problem of periodic gaits for the hybrid model of quadruped locomotion. We make use of the Poincar\'e sections analysis for the orbital stability of gaits.

Let $\varphi_{v}(t;x_{0},\xi_{v})$ denote the unique solution of the closed-loop ODE $\dot{x}=f_{v}^{\textrm{cl}}(x,\xi_{v})$, $v\in\mathcal{V}$ with the initial condition $x(0)=x_{0}$ for all $t\geq0$ in the maximal interval of existence. For the closed-loop system, since the control laws $u=\Gamma_{v}(x,\xi_{v})$ are already determined, one can analyze the discrete-time transitions $v\rightarrow\mu(v)$ on a set of reduced-order and parameterized switching manifolds $\hat{\mathcal{S}}_{v\rightarrow\mu(v)}(\xi_{v})\subset\mathcal{X}_{v}$ rather than the original ones $\mathcal{S}_{v\rightarrow\mu(v)}\subset\mathcal{X}_{v}\times\mathcal{U}_{v}$. Here, $\hat{\mathcal{S}}_{v\rightarrow\mu(v)}(\xi_{v})$ represents the set of all points $x\in\mathcal{X}_{v}$ for which $(x,u)=(x,\Gamma_{v}(x,\xi_{v}))\in\mathcal{S}_{v\rightarrow\mu(v)}$.

\begin{assumption}\label{regaularity}
For every $v\in\mathcal{V}$ and all $\xi_{v}\in\Xi_{v}$, $\hat{\mathcal{S}}_{v\rightarrow\mu(v)}(\xi_{v})$ is an embedded submanifold of $\mathcal{X}_{v}$ with the property $\textrm{dim}(\hat{\mathcal{S}}_{v\rightarrow\mu(v)}(\xi_{v}))=\textrm{dim}(\mathcal{X}_{v})-1$.
\end{assumption}

We now define the \textit{time-to-switching function} for the domain $v\in\mathcal{V}$ as the first time at which the state solution $\varphi_{v}(t;x_{0},\xi_{v})$ intersects the switching manifold $\hat{\mathcal{S}}_{v\rightarrow\mu(v)}$, i.e.,
$T_{v}\left(x_{0},\xi_{v}\right):=\inf\{t>0\,|\,\varphi_{v}(t;x_{0},\xi_{v})\in\hat{\mathcal{S}}_{v\rightarrow\mu(v)}(\xi_{v})\}$.  The \textit{generalized Poincar\'e map} for the domain $v\in\mathcal{V}$ is then defined as the flow of the closed-loop domain $v\in\mathcal{V}$ evaluated on $\hat{\mathcal{S}}_{v\rightarrow\mu(v)}(\xi_{v})$, i.e., $P_{v}:\mathcal{X}_{\mu^{-1}(v)}\times\Xi_{v}\rightarrow\hat{\mathcal{S}}_{v\rightarrow\mu(v)}$ and
\begin{equation*}
P_{v}\left(x,\xi_{v}\right)\!:=\!\varphi_{v}\left(T_{v}\left(\Delta_{\mu^{-1}(v)\rightarrow{v}}(x),\xi_{v}\right)\!;\!\Delta_{\mu^{-1}(v)\rightarrow{v}}(x),\xi_{v}\right)\!.
\end{equation*}
Next, let us take the parameters vector as $\xi:=\textrm{col}(\xi_{v_{1}},\xi_{v_{2}},\cdots,\xi_{v_{N}})\in\Xi\subset\Real^{p}$, where $N:=|\mathcal{V}|$ denotes the cardinal number of $\mathcal{V}$, $\Xi:=\Xi_{v_{1}}\times\Xi_{v_{2}}\times\cdots\Xi_{v_{N}}$, and $p:=\sum_{v\in\mathcal{V}} p_{v}$. The nominal parameters are also shown by $\xi^{\star}$. Suppose further that
$\omega:=\left\{v_{1},\mu\left(v_{1}\right),\mu^{2}\left(v_{1}\right),\cdots,\mu^{N-1}\left(v_{1}\right)\right\}$ represents the executed sequence of the vertices for the desired locomotion pattern $\mathcal{O}$, in which $\mu^{k}(v_{1}):=\mu(\mu^{k-1}(v_{1}))$ for $k=1,2,\cdots$ and $\mu^{0}(v_{1}):=v_{1}$. We remark that according to the periodicity of the desired gait, $\mu^{N}(v_{1})=v_{1}$. The \textit{full-order Poinca\'e return map} is finally taken as the composition of the generalized maps $P_{v}$ along the switching path $\omega$, i.e.,
\begin{alignat}{6}
& &&P\left(x,\xi\right):=\nonumber\\
& &&P_{v_{1}}\!\left(\!P_{\mu^{N-1}(v_{1})}\!\left(\cdots\!\left(P_{\mu(v_{1})}\left(x,\xi_{\mu(v_{1})}\right)\right)\!\cdots\!,\xi_{\mu^{N-1}(v_{1})}\right)\!,\xi_{v_{1}}\right)\!.\nonumber
\end{alignat}
The evolution of the closed-loop hybrid model on the Poincar\'e section $\hat{\mathcal{S}}_{v_{1}\rightarrow\mu(v_{1})}(\xi_{1})$ can then be described by the following discrete-time system
\begin{equation}\label{poincare_map}
x[k+1]=P\left(x[k],\xi\right),\quad k=0,1,2,\cdots,
\end{equation}
in which $k$ represents the step number.

\begin{assumption}[Transversality]\label{transversality_assumption}
We assume that the orbit $\mathcal{O}_{v}$ is transversal to the switching manifold $\hat{\mathcal{S}}_{v\rightarrow\mu(v)}(\xi_{v})$ for all $v\in\mathcal{V}$ and $\xi_{v}\in\Xi_{v}$. In particular, $\overline{\mathcal{O}}_{v}\cap\hat{\mathcal{S}}_{v\rightarrow\mu(v)}(\xi_{v})$ is a singleton for all controller parameters $\xi_{v}$. In our notation, $\overline{\mathcal{O}}_{v}$ denotes the set closure of $\mathcal{O}_{v}$.
\end{assumption}

\begin{theorem}[Poincar\'e Map]\label{Poincare_Map_Property}
Suppose Assumptions \ref{periodic_orbit_assumption}-\ref{transversality_assumption} are satisfied. Then, there exist (i) a subset $\hat{\Xi}\subset\Xi$ and (ii) a fixed point of $P(.,\xi):\hat{\mathcal{S}}_{v_{1}\rightarrow\mu(v_{1})}(\xi_{v_{1}})\rightarrow\hat{\mathcal{S}}_{v_{1}\rightarrow\mu(v_{1})}(\xi_{v_{1}})$, represented by $x_{1}^{\star}$, that is invariant under the choice of the controller parameters $\xi\in\hat{\Xi}$. That is, $x_{1}^{\star}\in\hat{\mathcal{S}}_{v_{1}\rightarrow\mu(v_{1})}(\xi_{v_{1}})$ and $P(x_{1}^{\star},\xi)=x_{1}^{\star}$ for all $\xi\in\hat{\Xi}$. Furthermore, $\Psi(\xi):=\frac{\partial P}{\partial x}(x_{1}^{\star},\xi)$ is well-defined and differentiable with respect to $\xi$ on $\hat{\Xi}$.
\end{theorem}

\begin{proof}
The proof is available online\footnote{\url{https://github.com/kavehakbarihamed/ProofofTh1/blob/master/proof.pdf}}.
\end{proof}

\begin{problem}[Asymptotic Stabilization]\label{asymptotic_stability_problem}
The asymptotic (exponential) stabilization problem consists of finding the controller parameters $\xi\in\hat{\Xi}$ such that the eigenvalues of the Jacobian matrix $\Psi(\xi)$ lie inside the unit circle.
\end{problem}

To solve Problem \ref{asymptotic_stability_problem}, one would need to find 1) a controller parameter $\xi\in\hat{\Xi}$, 2) a positive definite matrix $W$, and 3) a scalar $\gamma$ such that the increment of the Lyapunov function $V(\delta{}x):=\delta{}x^\top{}W^{-1}\delta{x}$ along the linearized discrete-time system $\delta{}x[k+1]=\Psi(\xi)\,\delta{}x[k]$ becomes negative definite, i.e., $\delta{}V[k]:=V[k+1]-V[k]<-\gamma\,V[k]$, where $\delta{}x[k]:=x[k]-x_{1}^{\star}$. This is equivalent to solving the following nonlinear matrix inequality (NMI)
\begin{equation}\label{NMI}
\begin{bmatrix}
W & \Psi(\xi)\,W\\
\star & (1-\gamma)\,W
\end{bmatrix}>0.
\end{equation}


\section{ITERATIVE BMI ALGORITHM}
\label{BMI ALGORITHM}
\vspace{-0.25em}

The objective of this section is to employ an iterative BMI algorithm to look for the controller parameters $\xi$ that solve the NMI \eqref{NMI}. The BMI algorithm was developed in \cite{Hamed_Buss_Grizzle_BMI_IJRR,Hamed_Gregg_decentralized_control_IEEE_CST} to systematically design centralized and decentralized nonlinear control algorithms for bipedal locomotion. Here, we employ the algorithm to design nonlinear controllers for higher-dimensional hybrid systems that describe quadruped locomotion. The objective of the BMI algorithm is to generate a sequence of controller parameters $\{\xi^{\ell}\}$ in an offline manner that would eventually solve the NMI \eqref{NMI}, where the superscript $\ell=0,1,\cdots$ represents the iteration number. The algorithm includes three main steps as follows.

\noindent\textit{Step 1 (Sensitivity Analysis):} During the iteration $\ell=0,1,\cdots$, the Jacobian matrix $\Psi(\xi^{\ell}+\Delta\xi)$ is approximated using the first-order Taylor series expansion, i.e.,
\begin{alignat}{6}
&\Psi\left(\xi^{\ell}+\Delta\xi\right)&&\approx{}\Psi\left(\xi^{\ell}\right)+\bar{\Psi}\left(\xi^{\ell}\right)\left(I\otimes\Delta\xi\right)=:\hat{\Psi}\left(\xi^{\ell},\Delta\xi\right)\nonumber.
\end{alignat}
Here, ``$\otimes$'' denotes the Kronecker product, $\bar{\Psi}(\xi^{\ell})$ represents the \textit{sensitivity matrix} during the iteration $\ell$, and $\hat{\Psi}(\xi^{\ell},\Delta\xi)$ denotes the first-order approximation of $\Psi(\xi^{\ell}+\Delta\xi)$ for sufficiently small $\Delta\xi\in\Real^{p}$. Our previous work has presented a systematic and effective numerical approach based on the variational equation in \cite[Theorems 1 and 2]{Hamed_Buss_Grizzle_BMI_IJRR} to compute the sensitivity matrix. We remark that the approximate Jacobian matrix $\hat{\Psi}(\xi^{\ell},\Delta\xi)$ is affine in terms of $\Delta\xi$ which will reduce the NMI \eqref{NMI} into a BMI in Step 2.

\noindent\textit{Step 2 (BMI Optimization):} The objective of Step 2 is to translate the NMI \eqref{NMI} into a BMI optimization problem that can be effectively solved using available solvers, e.g., PENBMI \cite{PENBMI} from TOMLAB \cite{TOMLAB}. In particular, we are interested in solving the following BMI optimization problem during the iteration $\ell$
\begin{alignat}{4}
&\min_{(W,\Delta\xi,\zeta,\gamma)}&&-w\,\gamma+\zeta\label{BMi_cost}\\
& \textrm{s.t.} &&\begin{bmatrix}
W & \hat{\Psi}\left(\xi^{\ell},\Delta\xi\right)W\\
\star & \left(1-\gamma\right)W
\end{bmatrix}>0\label{BMI}\\
& &&\begin{bmatrix}
I & \Delta\xi\\
\star & \zeta
\end{bmatrix}>0\label{LMI}\\
& &&\gamma>0\label{positive_gamma},
\end{alignat}
where \eqref{BMI} represents a BMI condition in terms of the decision variable. From Schur complement lemma, the linear matrix inequality (LMI) \eqref{LMI} introduces the dynamic upper bound $\zeta$ on $\|\Delta\xi\|_{2}^{2}$, i.e., $\zeta>\|\Delta\xi\|_{2}^{2}$. The cost function \eqref{BMi_cost} finally tries to minimize a linear combination of the convergence rate $\gamma$ and the dynamic bound $\zeta$. Here, $w>0$ is a weighting factor as  a tradeoff between improving the convergence rate or minimizing $\|\Delta\xi\|$ to have a good approximation based on the first-order Taylor series expansion.

\noindent\textit{Step 3 (Iteration):} Let $(W^{\star},\Delta\xi^{\star},\zeta^{\star},\gamma^{\star})$ represent a \textit{local} minimum (\textit{not} necessarily the global solution) for the BMI optimization problem \eqref{BMi_cost}-\eqref{positive_gamma}. Step 3 updates the controller parameters for the next iteration as $\xi^{\ell+1}=\xi^{\ell}+\Delta\xi^{\star}$. If the requirement of Problem \ref{asymptotic_stability_problem} is satisfied at $\xi=\xi^{\ell+1}$, the algorithm is successful and stops. Otherwise, it continuous by coming back to Step 1 around $\xi=\xi^{\ell+1}$ and going through the next steps. If the BMI problem in Step 2 is not feasible, the algorithm is not successful and stops. Sufficient conditions for the convergence of the algorithm to stabilizing solutions have been presented in \cite[Theorem 2]{Hamed_Gregg_decentralized_control_IEEE_CST}.


\section{NUMERICAL SIMULATIONS}
\label{NUMERICAL SIMULATIONS}
\vspace{-0.25em}

The objective of this section is to numerically validate the analytical results on a simulated model of the Vision $60$ robot.


\subsection{FROST}
\label{FROST}
\vspace{-0.25em}

We consider an amble gait for the robot as illustrated in Fig. \ref{hybrid_model_illustration}. To generate the gait, we make use of FROST (Fast Robot Optimization and Simulation Toolkit) --- an open-source MATLAB toolkit for developing model-based control and planning of dynamic legged locomotion \cite{FROTS}. FROST provides an efficient trajectory optimization framework for nonlinear hybrid dynamical systems. It uses the Hermite-Simpson collocation method to translate a trajectory planning problem into a traditional constrained nonlinear programming problem (NLP):
\begin{equation}
\textrm{argmin}\sum_{v\in\mathcal{V}} \sum_{i=1}^{N^{v}}\mathcal{L}_{v}(.)\, \delta t+E_{v}(.)
\end{equation}
subject to 1) equality constraints formed by the implicit Runge-Kutta method, and 2) inequality constraints arising from feasibility and physical limitations. In our notation, $\mathcal{L}_{v}(.)$ and $E_{v}(.)$ are the running and terminal costs for the domain $v\in\mathcal{V}$, respectively. In addition, $N^{v}$ represents the total number of grids for the domain $v$. FROST systematically translates the NLP problem into state-of-the-art solvers, such as IPOPT or SNOPT. In our problem, we considered a symmetric and periodic amble gait. The left-right symmetry reduces the motion planning problem of the eight-domain hybrid system into that of a four-domain hybrid system consisting of $344$ decision variables with $583$ constraints. Using a Ubuntu laptop with an i$7- 6820$HQ CPU $@\ 2.70$GHz and $16$GB RAM, it took $85$ seconds ($525$ iterations) for the FROST to find an optimal amble gait.


\subsection{BMI Algorithm}
\label{BMI Algorithm_num}
\vspace{-0.25em}

PENBMI is a general-purpose solver for BMIs which guarantees the convergence to a local optimal point satisfying the Karush Kuhn Tucker optimality conditions \cite{PENBMI}. To solve the BMI optimization problem \eqref{BMi_cost}-\eqref{positive_gamma}, we make use of the PENBMI solver from TOMLAB \cite{TOMLAB} integrated with the MATLAB environment through YALMIP \cite{YALMIP}.

\textit{Asymptotic Stabilization Problem:} Vision $60$ has $m=12$ actuators. For multi-contact domains $v\in\mathcal{V}$ of the amble gait, we consider a combination of relative degree one and two outputs whose holonomic portion is chosen as $10$ dimensional, i.e., $y_{v}:=\textrm{col}(y_{1v},y_{2v})\in\Real^{1+m_{2v}}$, $m_{2v}=10$, and $H_{2v}(\xi_{v})\in\Real^{10\times18}$. In particular, we have observed that for $m_{2v}=11$, the decoupling matrix $A_{v}$ in \eqref{decoupling_matrix_01} cannot be full-rank, and hence, we choose $m_{2v}=10$. For single-contact domains $v\in\mathcal{V}$, we take a $12$-dimensional output function to be regulated, i.e., $y_{v}:=y_{2v}\in\Real^{m_{2v}}$, $m_{2v}=m=12$, and $H_{2v}(\xi_{v})\in\Real^{12\times18}$. The phasing variable is also taken as the horizontal displacement of the robot along the waking direction (i.e., $x$-axis). We start with an initial set of output matrices $H_{2v}$, $v\in\mathcal{V}$ based on physical intuition. More specifically, for single-contact domains, we assume that the controlled variables are taken as the actuated body joints, i.e., $H_{2v}\,q=q_{\textrm{body}}$. For multi-contact domains, however, we remove the two rows of the $H_{2v}$ matrix that correspond to the knee and hip pitch angles of a contacting (i.e., stance) leg. The nonholonomic portion, i.e., $y_{1v}$, is then defined as \eqref{velocity_output} to regulate the forward velocity of the hip joint of the same leg. For this choice of virtual constraints, the dominant eigenvalues and spectral radius of the Jacobian matrix $\Psi(\xi^{0})$ become $\{-1.0350,-1.000,-0.1811\pm 0.8766i\}$ and $1.0350$, respectively. Hence, the gait is not stable. Starting with this set of parameters, we employ the iterative BMI algorithm of Section \ref{BMI ALGORITHM} with the weighting factor $w=0.1$. For the purpose of this paper, we only look for the optimal output matrix $H_{2v}$ during the domain $v_{1}$ with $10\times18=180$ parameters. We can apply this algorithm for computing other matrices as well. The algorithm successfully converges to a set of stabilizing parameters after four iterations, where each iteration takes approximately $30$ minutes on a Windows laptop with an i$9-8950$HK CPU $@\ 2.90$GHz and $32$GB RAM. For the BMI-optimized solution, the dominant eigenvalues and spectral radius of the Jacobian of the Poincar\'e map become $\{-0.9545\pm0.0600i,-0.2454\pm0.8600i\}$ and $0.9563$, respectively.
Figure \ref{Phase_Portraits} depicts the phase portraits during $150$ consecutive steps of walking. Convergence to a periodic orbit, even in the yaw position, can be seen in the figure. The animation of this simulation can be found at \cite{QuadrupedAnimation_Youtube}.

\begin{figure}[!t]
\centering
\subfloat{\includegraphics[width=1.65in]{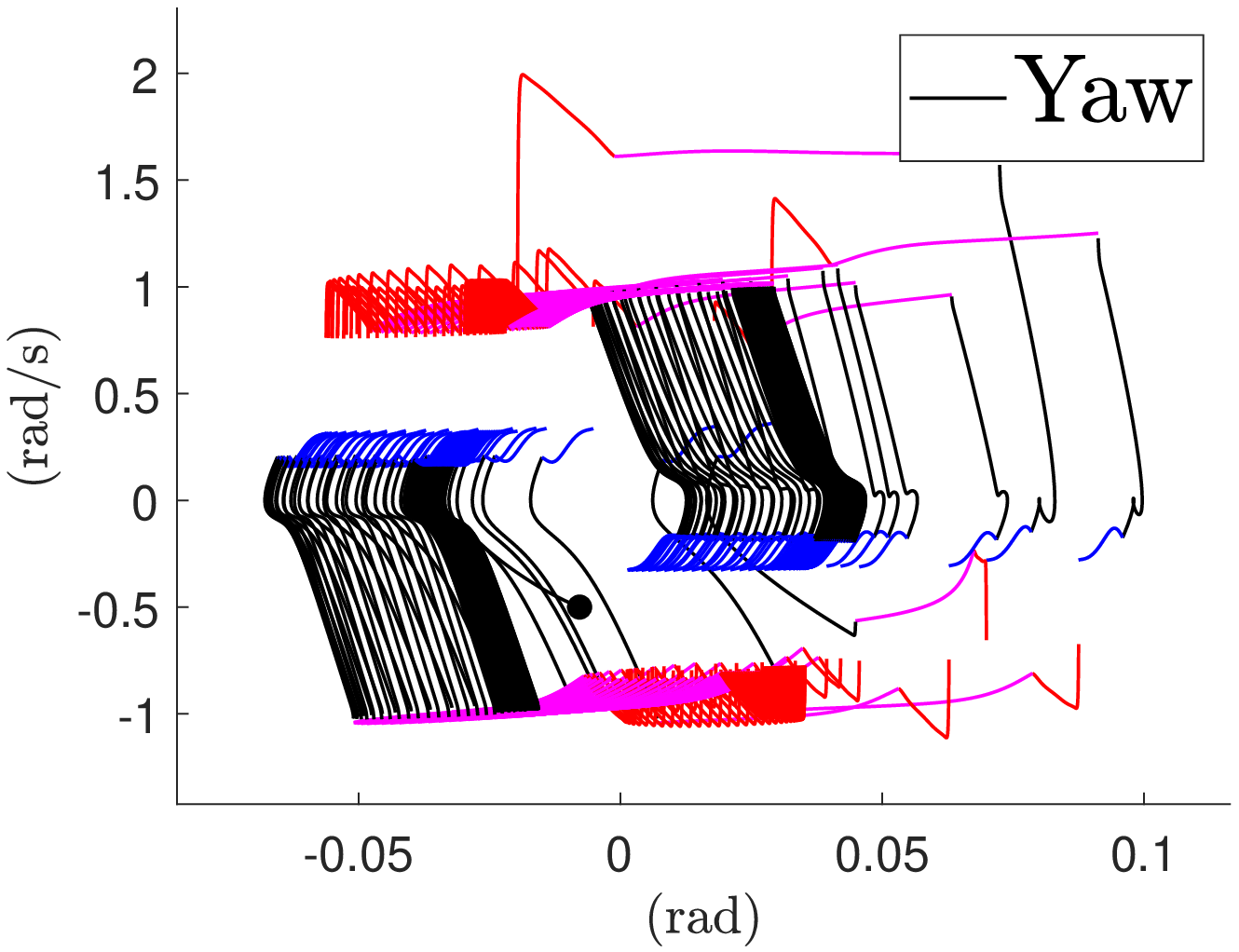}\label{yaw}}
\subfloat{\includegraphics[width=1.65in]{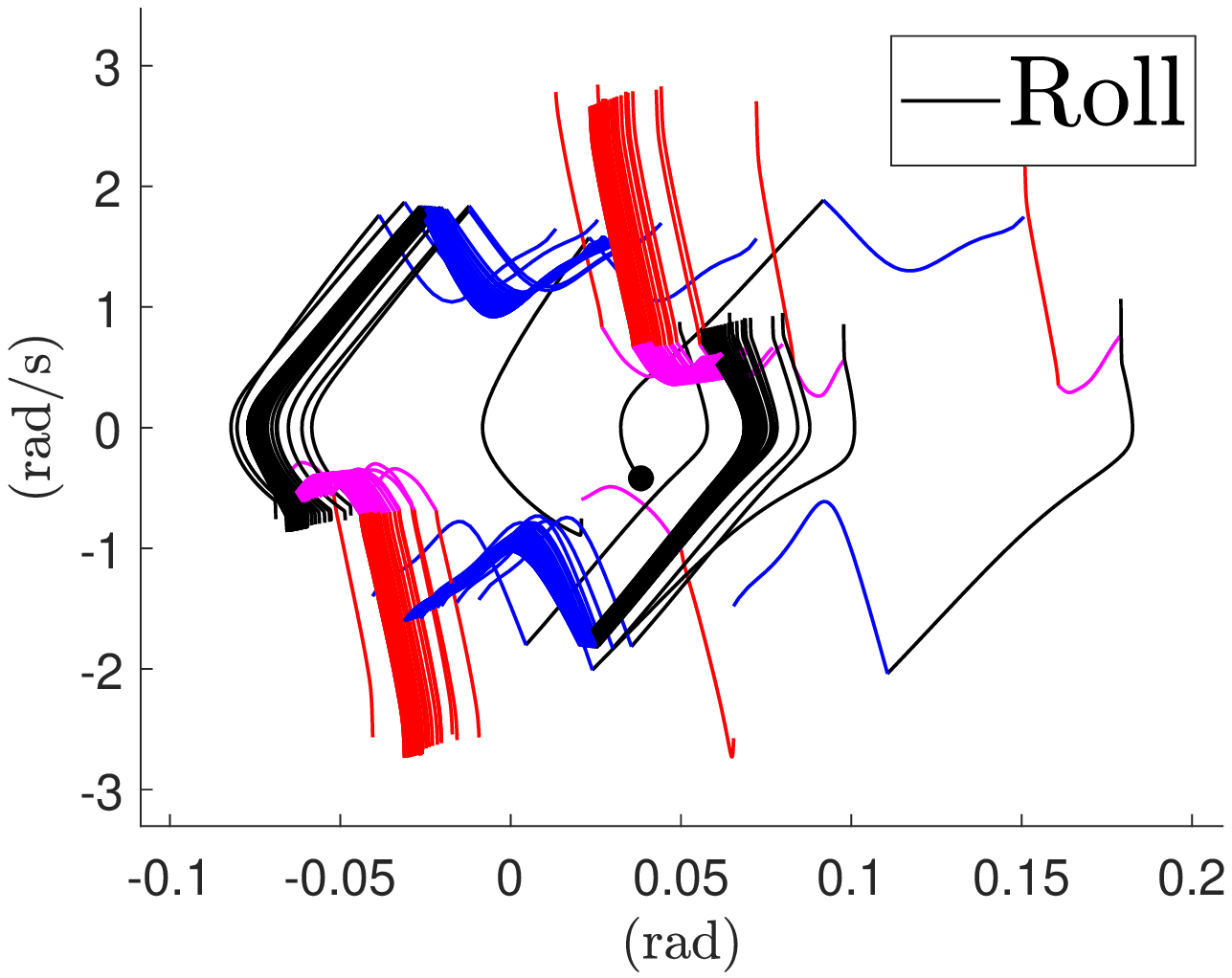}\label{roll}}\\
\vspace{-0.75em}
\subfloat{\includegraphics[width=1.65in]{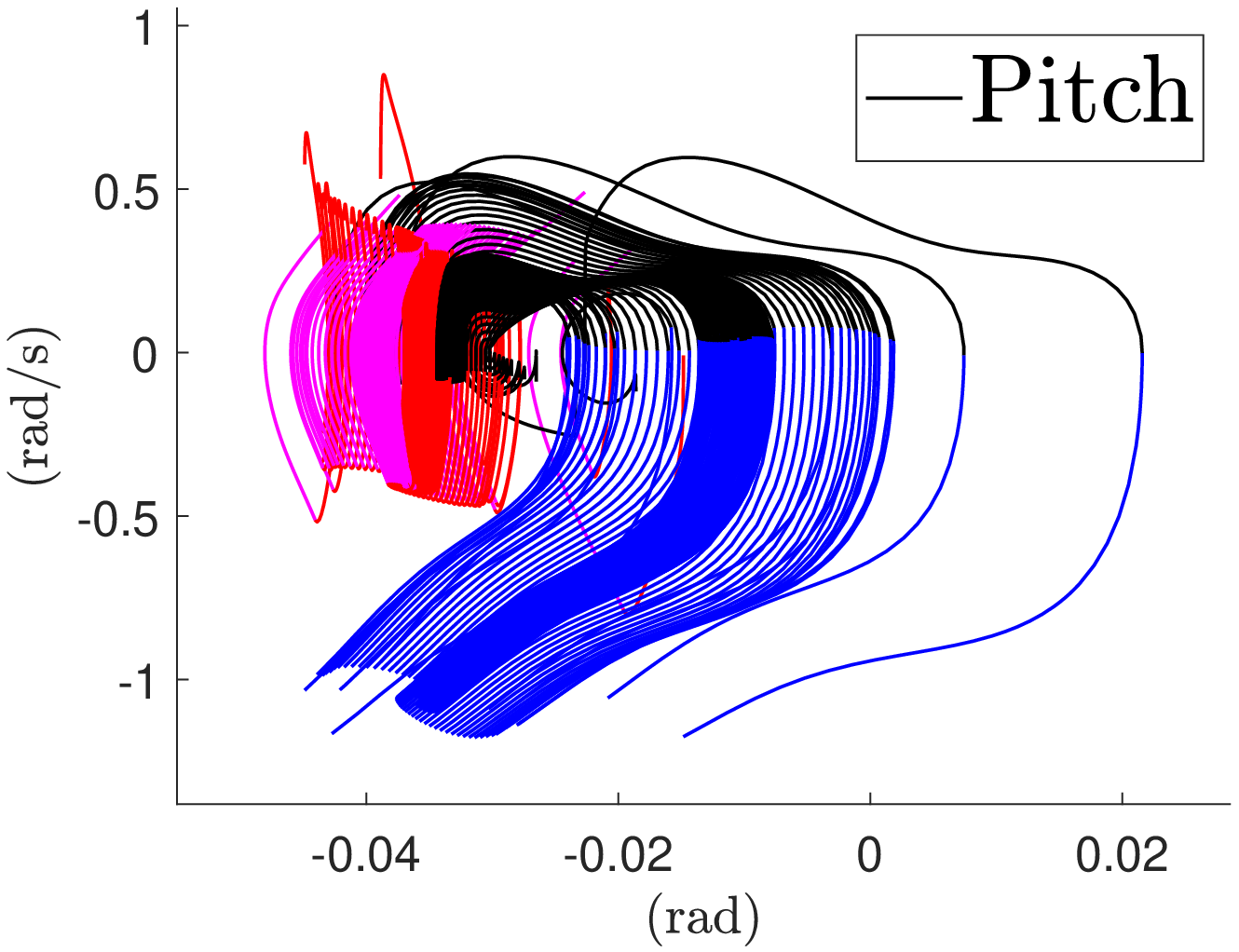}\label{pitch}}
\subfloat{\includegraphics[width=1.65in]{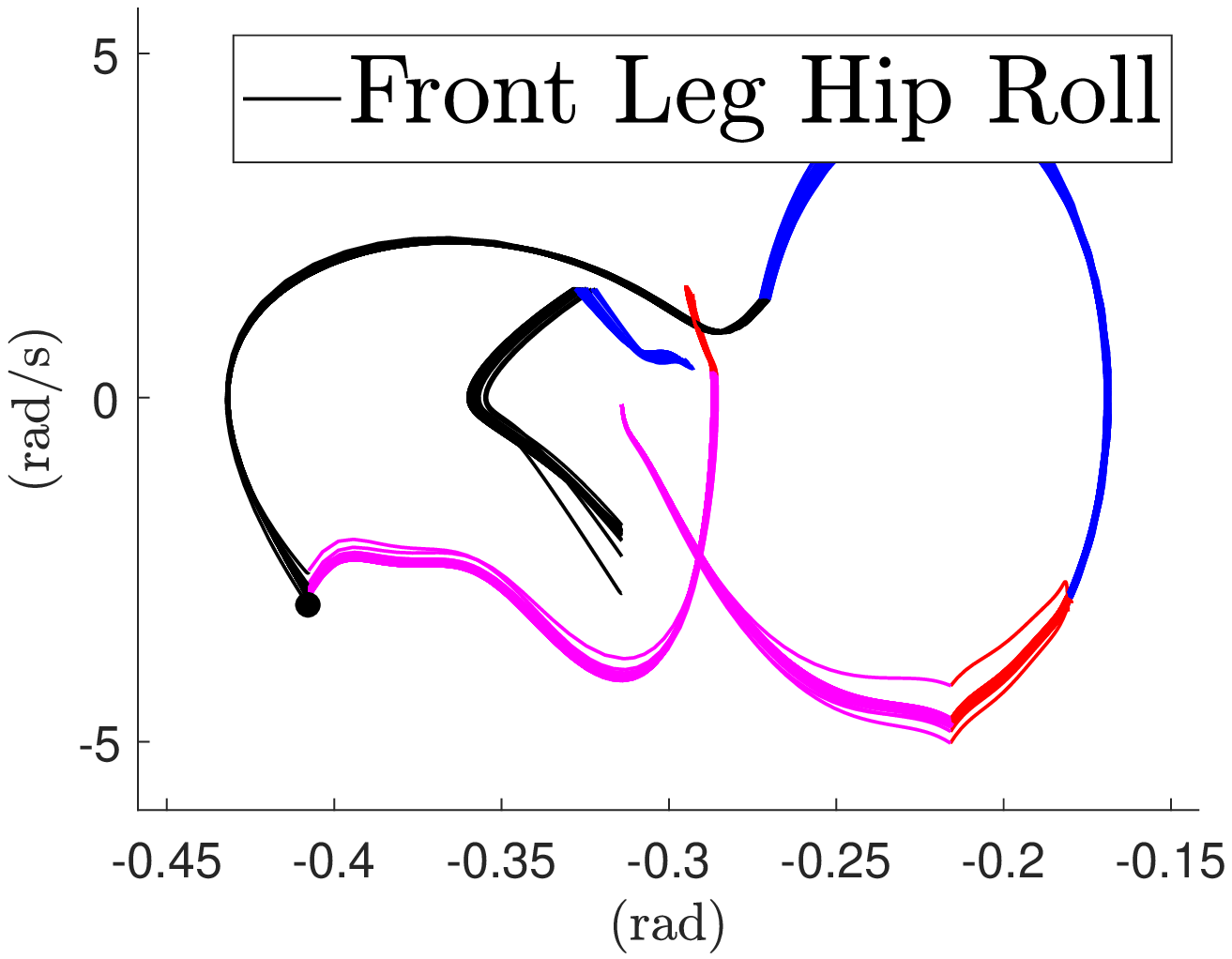}\label{fhr}}
\vspace{-0.75em}
\caption{Phase portraits during $150$ consecutive steps of $3$D quadruped walking with the BMI-optimized virtual constraint controllers.}
\label{Phase_Portraits}
\vspace{-1.5em}
\end{figure}


\vspace{-0.25em}
\section{CONCLUSIONS}
\label{Conclusions}
\vspace{-0.25em}

This paper presented an analytical foundation 1) to address multi-domain and high-dimensional hybrid models of quadruped robots, and 2) to systematically design HZD-based controllers to asymptotically stabilize dynamic gaits. We presented a family of parameterized nonlinear controllers for the single- and multi-contact domains of quadruped locomotion. The controllers zero a combination of holonomic and nonholonomic outputs. We investigated the properties of the parameterized and high-dimensional Poincar\'e map for the full-order closed-loop hybrid system. The asymptotic stabilization problem of multi-domain gaits was then translated into an iterative BMI optimization algorithm that can be effectively solved using available software packages. To demonstrate the power of the analytical framework, an optimal amble gait was designed using the FROST toolkit for the Vision $60$ robot with $36$ state variables and $12$ control inputs. To stabilize the gait, the iterative BMI algorithm was successfully employed to design a set of HZD-based controllers with $180$ controller parameters.

For future research, we will use this framework to design and experimentally implement stabilizing centralized as well as decentralized controllers for different gait patterns of quadruped locomotion and the Vision $60$ robot.





\begin{thebibliography}{10}
\providecommand{\url}[1]{#1}
\csname url@samestyle\endcsname
\providecommand{\newblock}{\relax}
\providecommand{\bibinfo}[2]{#2}
\providecommand{\BIBentrySTDinterwordspacing}{\spaceskip=0pt\relax}
\providecommand{\BIBentryALTinterwordstretchfactor}{4}
\providecommand{\BIBentryALTinterwordspacing}{\spaceskip=\fontdimen2\font plus
\BIBentryALTinterwordstretchfactor\fontdimen3\font minus
  \fontdimen4\font\relax}
\providecommand{\BIBforeignlanguage}[2]{{%
\expandafter\ifx\csname l@#1\endcsname\relax
\typeout{** WARNING: IEEEtran.bst: No hyphenation pattern has been}%
\typeout{** loaded for the language `#1'. Using the pattern for}%
\typeout{** the default language instead.}%
\else
\language=\csname l@#1\endcsname
\fi
#2}}
\providecommand{\BIBdecl}{\relax}
\BIBdecl

\bibitem{Grizzle_Asymptotically_Stable_Walking_IEEE_TAC}
J.~Grizzle, G.~Abba, and F.~Plestan, ``Asymptotically stable walking for biped
  robots: analysis via systems with impulse effects,'' \emph{Automatic Control,
  IEEE Transactions on}, vol.~46, no.~1, pp. 51--64, Jan 2001.

\bibitem{Chevallereau_Grizzle_3D_Biped_IEEE_TRO}
C.~Chevallereau, J.~Grizzle, and C.-L. Shih, ``Asymptotically stable walking of
  a five-link underactuated 3-{D} bipedal robot,'' \emph{Robotics, IEEE
  Transactions on}, vol.~25, no.~1, pp. 37--50, Feb 2009.

\bibitem{Ames_RES_CLF_IEEE_TAC}
A.~Ames, K.~Galloway, K.~Sreenath, and J.~Grizzle, ``Rapidly exponentially
  stabilizing control {Lyapunov} functions and hybrid zero dynamics,''
  \emph{Automatic Control, IEEE Transactions on}, vol.~59, no.~4, pp. 876--891,
  April 2014.

\bibitem{Spong_Controlled_Symmetries_IEEE_TAC}
M.~Spong and F.~Bullo, ``Controlled symmetries and passive walking,''
  \emph{Automatic Control, IEEE Transactions on}, vol.~50, no.~7, pp.
  1025--1031, July 2005.

\bibitem{Spong_Passivity_IEEE_RAM}
M.~Spong, J.~Holm, and D.~Lee, ``Passivity-based control of bipedal
  locomotion,'' \emph{Robotics Automation Magazine, IEEE}, vol.~14, no.~2, pp.
  30--40, June 2007.

\bibitem{Manchester_Tedrake_LQR_IJRR}
I.~Manchester, U.~Mettin, F.~Iida, and R.~Tedrake, ``Stable dynamic walking
  over uneven terrain,'' \emph{The International Journal of Robotics Research},
  vol.~30, no.~3, pp. 265--279, 2011.

\bibitem{Tedrake_L2_gain_ICRA}
H.~Dai and R.~Tedrake, ``$\mathcal{L}_{2}$-gain optimization for robust bipedal
  walking on unknown terrain,'' in \emph{Robotics and Automation, IEEE
  International Conference on}, May 2013, pp. 3116--3123.

\bibitem{Guobiao_Zefran_IEEE_ICRA}
G.~Song and M.~Zefran, ``Underactuated dynamic three-dimensional bipedal
  walking,'' in \emph{Robotics and Automation. Proceedings IEEE International
  Conference on}, May 2006, pp. 854--859.

\bibitem{Gregg_3D_Biped_IEEE_TRO}
R.~Gregg, A.~Tilton, S.~Candido, T.~Bretl, and M.~Spong, ``Control and planning
  of {3-D} dynamic walking with asymptotically stable gait primitives,''
  \emph{Robotics, IEEE Transactions on}, vol.~28, no.~6, pp. 1415--1423, Dec
  2012.

\bibitem{Gregg_Controlled_Reduction_IEEE_TAC}
R.~Gregg and L.~Righetti, ``Controlled reduction with unactuated cyclic
  variables: Application to {3D} bipedal walking with passive yaw rotation,''
  \emph{Automatic Control, IEEE Transactions on}, vol.~58, no.~10, pp.
  2679--2685, Oct 2013.

\bibitem{Byl_Tedrake_Approximate_optimal_control_ICRA}
K.~Byl and R.~Tedrake, ``Approximate optimal control of the compass gait on
  rough terrain,'' in \emph{Robotics and Automation. IEEE International
  Conference on}, May 2008, pp. 1258--1263.

\bibitem{Hamed_Gregg_decentralized_control_IEEE_CST}
K.~Akbari~Hamed and R.~D. Gregg, ``Decentralized feedback controllers for
  robust stabilization of periodic orbits of hybrid systems: Application to
  bipedal walking,'' \emph{Control Systems Technology, IEEE Transactions on},
  vol.~25, no.~4, pp. 1153--1167, July 2017.

\bibitem{Hamed_Grizzle_Event_Based_IEEE_TRO}
K.~Akbari~Hamed and J.~Grizzle, ``Event-based stabilization of periodic orbits
  for underactuated {3-D} bipedal robots with left-right symmetry,''
  \emph{Robotics, IEEE Transactions on}, vol.~30, no.~2, pp. 365--381, April
  2014.

\bibitem{Cheavallereau_Grizzle_RABBIT}
C.~Chevallereau, G.~Abba, Y.~Aoustin, F.~Plestan, E.~Westervelt, C.~Canudas-de
  Wit, and J.~Grizzle, ``{RABBIT}: a testbed for advanced control theory,''
  \emph{Control Systems Magazine, IEEE}, vol.~23, no.~5, pp. 57--79, Oct 2003.

\bibitem{Morriz_Grizzle_Hybrid_Invariant_Manifolds_IEEE_TAC}
B.~Morris and J.~Grizzle, ``Hybrid invariant manifolds in systems with impulse
  effects with application to periodic locomotion in bipedal robots,''
  \emph{Automatic Control, IEEE Transactions on}, vol.~54, no.~8, pp.
  1751--1764, Aug 2009.

\bibitem{Poulakakis_Grizzle_SLIP_IEEE_TAC}
I.~Poulakakis and J.~Grizzle, ``The spring loaded inverted pendulum as the
  hybrid zero dynamics of an asymmetric hopper,'' \emph{Automatic Control, IEEE
  Transactions on}, vol.~54, no.~8, pp. 1779--1793, Aug 2009.

\bibitem{Sreenath_Grizzle_HZD_Walking_IJRR}
K.~Sreenath, H.-W. Park, I.~Poulakakis, and J.~W. Grizzle, ``Compliant hybrid
  zero dynamics controller for achieving stable, efficient and fast bipedal
  walking on {MABEL},'' \emph{The International Journal of Robotics Research},
  vol.~30, no.~9, pp. 1170--1193, Aug. 2011.

\bibitem{Collins_Ruina_Tedrake_Wisse_Efficient_Bipedal_Robots}
S.~Collins, A.~Ruina, R.~Tedrake, and M.~Wisse, ``Efficient bipedal robots
  based on passive-dynamic walkers,'' \emph{Science}, vol. 307, no. 5712, pp.
  1082--1085, 2005.

\bibitem{Tedrake_Stochastic_policy_gradient_IROS}
R.~Tedrake, T.~Zhang, and H.~Seung, ``{Stochastic policy gradient reinforcement
  learning on a simple 3D biped},'' in \emph{Intelligent Robots and Systems.
  Proceedings 2004 IEEE/RSJ International Conference on}, vol.~3, Sept 2004,
  pp. 2849--2854 vol.3.

\bibitem{Byl01082009}
K.~Byl and R.~Tedrake, ``Metastable walking machines,'' vol.~28, no.~8, pp.
  1040--1064, 2009.

\bibitem{Johnson_Burden_Koditschek}
A.~M. Johnson, S.~A. Burden, and D.~E. Koditschek, ``A hybrid systems model for
  simple manipulation and self-manipulation systems,'' \emph{The International
  Journal of Robotics Research}, vol.~35, no.~11, pp. 1354--1392, 2016.

\bibitem{Burden_SIAM}
S.~A. Burden, S.~S. Sastry, D.~E. Koditschek, and S.~Revzen, ``Event--selected
  vector field discontinuities yield piecewise--differentiable flows,''
  \emph{SIAM Journal on Applied Dynamical Systems}, vol.~15, no.~2, pp.
  1227--1267, 2016.

\bibitem{Vasudevan2017}
R.~Vasudevan, \emph{Hybrid System Identification via Switched System Optimal
  Control for Bipedal Robotic Walking}.\hskip 1em plus 0.5em minus 0.4em\relax
  Cham: Springer International Publishing, 2017, pp. 635--650.

\bibitem{Park_Cheetah}
H.-W. Park, P.~M. Wensing, and S.~Kim, ``{High-speed bounding with the MIT
  Cheetah 2: Control design and experiments},'' \emph{The International Journal
  of Robotics Research}, vol.~36, no.~2, pp. 167--192, 2017.

\bibitem{Ioannis_ISS_2017}
S.~Veer, Rakesh, and I.~Poulakakis, ``Input-to-state stability of periodic
  orbits of systems with impulse effects via poincar\'e analysis,''
  \emph{{arXiv preprint arXiv:1712.03291}}, 2017.

\bibitem{Ames_HybridReduction_Original_Paper}
A.~D. Ames, R.~D. Gregg, E.~D. Wendel, and S.~Sastry, ``On the geometric
  reduction of controlled three-dimensional bipedal robotic walkers,'' in
  \emph{Lagrangian and Hamiltonian Methods for Nonlinear Control 2006}.\hskip
  1em plus 0.5em minus 0.4em\relax Berlin, Heidelberg: Springer Berlin
  Heidelberg, 2007, pp. 183--196.

\bibitem{Ames_Sastry_Hybrid_Reduction_CDC}
A.~Ames and S.~Sastry, ``Hybrid geometric reduction of hybrid systems,'' in
  \emph{Decision and Control, 45th IEEE Conference on}, Dec 2006, pp. 923--929.

\bibitem{Ames_Book_Chapter}
A.~Ames, R.~W. Sinnet, and E.~D. Wendel,
  ``\BIBforeignlanguage{English}{Three-dimensional kneed bipedal walking: A
  hybrid geometric approach},'' in \emph{\BIBforeignlanguage{English}{Hybrid
  Systems: Computation and Control}}, ser. Lecture Notes in Computer Science,
  R.~Majumdar and P.~Tabuada, Eds.\hskip 1em plus 0.5em minus 0.4em\relax
  Springer Berlin Heidelberg, 2009, vol. 5469, pp. 16--30.

\bibitem{Gregg_Spong_IJRR}
R.~D. Gregg and M.~W. Spong, ``Reduction-based control of three-dimensional
  bipedal walking robots,'' \emph{The International Journal of Robotics
  Research}, vol.~29, no.~6, pp. 680--702, May 2010.

\bibitem{Shiriaev_Transverse-Linearization_IEEE_TAC}
A.~Shiriaev, L.~Freidovich, and S.~Gusev, ``Transverse linearization for
  controlled mechanical systems with several passive degrees of freedom,''
  \emph{Automatic Control, IEEE Transactions on}, vol.~55, no.~4, pp. 893--906,
  April 2010.

\bibitem{Westervelt_Grizzle_Koditschek_HZD_IEEE_TRO}
E.~Westervelt, J.~Grizzle, and D.~Koditschek, ``Hybrid zero dynamics of planar
  biped walkers,'' \emph{Automatic Control, IEEE Transactions on}, vol.~48,
  no.~1, pp. 42--56, Jan 2003.

\bibitem{FRI}
A.~Goswami, ``{Postural Stability of Biped Robots and the Foot-Rotation
  Indicator (FRI) Point},'' \emph{The International Journal of Robotics
  Research}, vol.~18, no.~6, pp. 523--533, 1999.

\bibitem{Vukobratovic_Book}
M.~Vukobratovi\'c, B.~Borovac, and D.~Surla, \emph{Dynamics of Biped
  Locomotion}.\hskip 1em plus 0.5em minus 0.4em\relax Springer, 1990.

\bibitem{Maggiore_Virtual_Constraints_IEEE_TAC}
M.~Maggiore and L.~Consolini, ``Virtual holonomic constraints for {Euler
  Lagrange} systems,'' \emph{Automatic Control, IEEE Transactions on}, vol.~58,
  no.~4, pp. 1001--1008, April 2013.

\bibitem{Mohammadi_Automatica}
A.~Mohammadi, M.~Maggiore, and L.~Consolini, ``Dynamic virtual holonomic
  constraints for stabilization of closed orbits in underactuated mechanical
  systems,'' \emph{Automatica}, vol.~94, pp. 112 -- 124, 2018.

\bibitem{Isidori_Book}
A.~Isidori, \emph{Nonlinear Control Systems}.\hskip 1em plus 0.5em minus
  0.4em\relax Springer; 3rd edition, 1995.

\bibitem{Martin_Schmiedeler_IJRR}
A.~E. Martin, D.~C. Post, and J.~P. Schmiedeler, ``{The effects of foot
  geometric properties on the gait of planar bipeds walking under HZD-based
  control},'' \emph{The International Journal of Robotics Research}, vol.~33,
  no.~12, pp. 1530--1543, 2014.

\bibitem{Ames_DURUS_TRO}
A.~Hereid, C.~M. Hubicki, E.~A. Cousineau, and A.~D. Ames, ``Dynamic humanoid
  locomotion: A scalable formulation for {HZD} gait optimization,'' \emph{IEEE
  Transactions on Robotics}, pp. 1--18, 2018.

\bibitem{Ramezani_Hurst_Hamed_Grizzle_ATRIAS_ASME}
A.~Ramezani, J.~Hurst, K.~Akbai~Hamed, and J.~Grizzle, ``Performance analysis
  and feedback control of {ATRIAS}, a three-dimensional bipedal robot,''
  \emph{Journal of Dynamic Systems, Measurement, and Control December, ASME},
  vol. 136, no.~2, December 2013.

\bibitem{Park_Grizzle_Finite_State_Machine_IEEE_TRO}
H.-W. Park, A.~Ramezani, and J.~Grizzle, ``A finite-state machine for
  accommodating unexpected large ground-height variations in bipedal robot
  walking,'' \emph{Robotics, IEEE Transactions on}, vol.~29, no.~2, pp.
  331--345, April 2013.

\bibitem{Byl_HZD}
C.~O. Saglam and K.~Byl, ``Meshing hybrid zero dynamics for rough terrain
  walking,'' in \emph{2015 IEEE International Conference on Robotics and
  Automation (ICRA)}, May 2015, pp. 5718--5725.

\bibitem{Gregg_Virtual_Constraints_Powered_Prosthetic_IEEE_TRO}
R.~Gregg, T.~Lenzi, L.~Hargrove, and J.~Sensinger, ``Virtual constraint control
  of a powered prosthetic leg: From simulation to experiments with transfemoral
  amputees,'' \emph{Robotics, IEEE Transactions on}, vol.~30, no.~6, pp.
  1455--1471, Dec 2014.

\bibitem{multicontactprosthetic_02}
H.~Zhao, J.~Horn, J.~Reher, V.~Paredes, and A.~Ames, ``Multicontact locomotion
  on transfemoral prostheses via hybrid system models and optimization-based
  control,'' \emph{IEEE Transactions on Automation Science and Engineering},
  vol.~13, no.~2, pp. 502--513, April 2016.

\bibitem{Grizzle_Decentralized}
A.~Agrawal, O.~Harib, A.~Hereid, S.~Finet, M.~Masselin, L.~Praly, A.~Ames,
  K.~Sreenath, and J.~Grizzle, ``First steps towards translating {HZD} control
  of bipedal robots to decentralized control of exoskeletons,'' \emph{IEEE
  Access}, vol.~5, pp. 9919--9934, 2017.

\bibitem{Ioannis2016}
Q.~Cao and I.~Poulakakis, ``Quadrupedal running with a flexible torso: control
  and speed transitions with sums-of-squares verification,'' \emph{Artificial
  Life and Robotics}, vol.~21, no.~4, pp. 384--392, Dec 2016.

\bibitem{Ghost_Robotics}
(2018) {Ghost Robotics}, {https://www.ghostrobotics.io/}.

\bibitem{Hamed_Buss_Grizzle_BMI_IJRR}
K.~Akbari~Hamed, B.~Buss, and J.~Grizzle, ``Exponentially stabilizing
  continuous-time controllers for periodic orbits of hybrid systems:
  Application to bipedal locomotion with ground height variations,'' \emph{The
  International Journal of Robotics Research}, vol.~35, no.~8, pp. 977--999,
  2016.

\bibitem{Haddad_Hybrid_Book}
W.~Haddad, V.~Chellaboina, and S.~Nersesov, \emph{Impulsive and Hybrid
  Dynamical Systems: Stability, Dissipativity, and Control}.\hskip 1em plus
  0.5em minus 0.4em\relax Princeton University Press, July 2006.

\bibitem{Hamed_Grizzle_NAHS}
K.~Akbari~Hamed and J.~W. Grizzle, ``Reduced-order framework for exponential
  stabilization of periodic orbits on parameterized hybrid zero dynamics
  manifolds: Application to bipedal locomotion,'' \emph{Nonlinear Analysis:
  Hybrid Systems}, vol.~25, pp. 227--245, August 2017.

\bibitem{FROTS}
A.~Hereid and A.~D. Ames, ``{FROST}: Fast robot optimization and simulation
  toolkit,'' in \emph{Intelligent Robots and Systems (IROS), 2017 IEEE/RSJ
  International Conference on}.\hskip 1em plus 0.5em minus 0.4em\relax IEEE,
  2017, pp. 4552--4559.

\bibitem{Hurmuzlu_Impact}
Y.~Hurmuzlu and D.~B. Marghitu, ``Rigid body collisions of planar kinematic
  chains with multiple contact points,'' vol.~13, no.~1, pp. 82--92, 1994.

\bibitem{Gregg_Toward_Biomimetic_Control_IEEE_CST}
R.~Gregg and J.~Sensinger, ``Towards biomimetic virtual constraint control of a
  powered prosthetic leg,'' \emph{Control Systems Technology, IEEE Transactions
  on}, vol.~22, no.~1, pp. 246--254, Jan 2014.

\bibitem{Ames_Nonholonomic}
A.~D. Ames, ``First steps toward automatically generating bipedal robotic
  walking from human data,'' in \emph{Robot Motion and Control 2011},
  K.~Koz{\l}owski, Ed.\hskip 1em plus 0.5em minus 0.4em\relax London: Springer
  London, 2012, pp. 89--116.

\bibitem{Ames_Human_Inspired_IEE_TAC}
A.~Ames, ``Human-inspired control of bipedal walking robots,'' \emph{Automatic
  Control, IEEE Transactions on}, vol.~59, no.~5, pp. 1115--1130, May 2014.

\bibitem{PENBMI}
D.~Henrion, J.~Lofberg, M.~Kocvara, and M.~Stingl, ``{Solving polynomial static
  output feedback problems with PENBMI},'' in \emph{Decision and Control, and
  European Control Conference. 44th IEEE Conference on}, Dec 2005, pp.
  7581--7586.

\bibitem{TOMLAB}
{TOMLAB} optimization, {http://tomopt.com/tomlab/}.

\bibitem{YALMIP}
J.~Lofberg, ``{YALMIP}: a toolbox for modeling and optimization in {MATLAB},''
  in \emph{Computer Aided Control Systems Design, 2004 IEEE International
  Symposium on}, Sept 2004, pp. 284--289.

\bibitem{QuadrupedAnimation_Youtube}
(2018) {Quadruped Walking Animation}, {https://youtu.be/hikx2eWa1C4}.

\end{thebibliography}

\end{document}